\documentclass{amsart}
\usepackage{tikz}
\usepackage{hyperref}
\hypersetup{
 pdftitle={Real Bott manifolds and Acyclic digraphs},
pdfauthor={Suyoung Choi and Sang-il Oum}
}

\usepackage{enumerate}
\theoremstyle{ams}
\newtheorem{theorem}{Theorem}[section]
\newtheorem{proposition}[theorem]{Proposition}
\newtheorem{lemma}[theorem]{Lemma}

\theoremstyle{definition}

\newtheorem{remark}[theorem]{Remark}

\newcommand{\Z}{\mathbb{Z}}
\newcommand{\R}{\mathbb{R}}
\newcommand{\RP}{\mathbb{R}P}

\newcommand\slide\Delta
\DeclareMathOperator{\rank}{rank}

\begin{document}
\title[Real Bott manifolds and Acyclic Digraphs] {Real Bott manifolds and Acyclic digraphs}

\author{Suyoung Choi}
\address{Department of Mathematics, Osaka City University, Sugimoto, Sumiyoshi-ku, Osaka 558-8585, Japan}
\email{choi@sci.osaka-cu.ac.jp}

\author{Sang-il Oum}
\address{Department of Mathematical Sciences, KAIST, 335 Gwahangno, Yuseong-gu, Daejeon 305-701, Republic of Korea}
\email{sangil@kaist.edu}

\thanks{The first author was supported by the Japanese Society for the
  Promotion of Sciences (JSPS grant no. P09023).
  The second author was
   supported by Basic Science Research Program
     through the National Research Foundation of Korea (NRF)
     funded by the Ministry of Education, Science and Technology
     (2009-0063183)
  and TJ Park Junior Faculty Fellowship.
}


\date{\today}
\maketitle

\begin{abstract}
  Masuda (2008) provided the characterization of real Bott manifolds in terms of three operations on upper triangular matrices.
  We provide a combinatorial characterization of real Bott manifolds
  up to diffeomorphism in terms of operations on  directed acyclic graphs.
  Our observation leads to several new invariants of real Bott manifolds.
\end{abstract}

\section{Introduction}
A \emph{small cover}, defined by Davis and Januszkiewicz~\cite{Da-Ja-1991},
is an $n$-dimensional closed smooth manifold $M$
with a smooth action of real torus $(S^0)^n(=:T)$ satisfying the following two conditions:
\begin{itemize}
  \item the action is locally isomorphic to a standard action of $T$ on $\R^n$, and
  \item the orbit space $M/T$ can be identified with a simple (combinatorial) polytope $P$.
\end{itemize}
A \emph{cube} is a polytope combinatorially equivalent to the cartesian product of finitely many intervals. We will restrict our attention to the case where $P$ is a cube. In this case, a small cover is said to be \emph{over} a cube.

Small covers over cubes are known as \emph{real Bott manifolds}, see \cite{Ch-Ma-Su-2010-2} or Section~\ref{sec:small cover}. They are obtained as iterated $\RP^1$ bundles starting with a point, where each stage is the projectivization of a Whitney sum of two real line bundles. The topological classification of real Bott manifolds is known by Kamishima and Masuda~\cite{Ka-Ma-2009}, and Masuda~\cite{Masuda-2008}; two real Bott manifolds are diffeomorphic if and only if their cohomology rings with $\Z_2$ coefficients are isomorphic as graded rings.
Choi~\cite{Choi-2008} showed that small covers over cubes are strongly related to acyclic digraphs.

We introduce two operations on acyclic digraphs and show that all diffeomorphism types of real Bott manifolds can be described by the composition of these two operations.
This combinatorial observation allows us to invent several invariants of real Bott manifolds.

Our combinatorial interpretation allows more efficient enumeration of all $n$-dimensional real Bott manifolds up to diffeomorphism.
We list the number $B_n$ of $n$-dimensional real Bott manifolds in
Table~\ref{tab:Bott equivalence} for $n \leq 8$. Previously, $B_n$
were computed only for $n\leq 5$ in \cite{Masuda-2008} and
\cite{Naz-08}, but it was a hard task even for $n=5$ when using topological or linear algebraic methods.
The classification for $8$-dimensional real Bott manifolds takes less than $10$ minutes on a regular desktop computer,
when using the list of non-isomorphic acyclic digraphs
provided by B. D. McKay.\footnote{\url{http://cs.anu.edu.au/~bdm/data/digraphs.html}}
In addition, we also list the numbers $E_n$ and $S_n$ of $n$-dimensional orientable real Bott manifolds and symplectic real Bott manifolds in Table~\ref{tab:Bott equivalence} for $n \leq 8$, respectively.

\begin{table}
  \centering
\begin{tabular}{c|c c c c c c c c }
  \hline
  $n$ & 1 & 2 & 3 & 4  & 5 & 6 & 7 & 8\\ \hline
  $B_n $ & 1 & 2 & 4 & 12 & 54 & 472 & 8,512 & 328,416\\ \hline
  $E_n $ & 1 & 1 & 2 & 3 & 8 & 29 & 222 & 3,607\\ \hline
  $S_n $ &   & 1 &   & 2 &   & 6  &     &  31 \\
  \hline
\end{tabular}
\caption{The numbers $B_n, E_n, S_n$ of $n$-dimensional real Bott manifolds, orientable real Bott manifolds and symplectic real Bott manifolds up to diffeomorphism, respectively.}
\label{tab:Bott equivalence}
\end{table}

The rest of this paper is organized as follows.  In Sections~\ref{section:acyclic digraphs} and \ref{sec:small cover}, we briefly review some notions in acyclic digraphs and small covers respectively. In Section~\ref{sec:classification}, we classify the acyclic digraph up to Bott equivalence, which implies the diffeomorphism class of real Bott manifolds. Finally, in Section~\ref{sec:invariant}, we introduce several invariants of acyclic digraphs up to Bott equivalence.

\section{Acyclic digraphs} \label{section:acyclic digraphs}
\subsection{Preliminaries on acyclic digraphs}
We briefly review the terminologies in graph theory.
A \emph{directed graph},  or simply a \emph{digraph} $D$
is a pair $(V,E)$ consisting of a finite set $V$ of \emph{vertices}
and a set $E$ of ordered pairs
$e=(u,v)$ of distinct vertices of $D$, called
\emph{arcs}.
We write $V(D)$ to denote the set of all vertices of $D$.
An ordering $v_1,v_2,\ldots,v_n$ of vertices of a digraph $D$ is an \emph{acyclic ordering} if $i<j$ for each arc $(v_i,v_j)$ in $D$.
A digraph is \emph{acyclic} if it admits an acyclic ordering.

A vertex $u$ is \emph{adjacent} to a vertex $v$ in $D$ if $(u,v)$ is an arc of $D$.
If $(u,v)$ is an arc of $D$, then $v$ is called an \emph{out-neighbor} of $u$ and $u$ is called an \emph{in-neighbor} of $v$.
We write $N_D^+(v)$, $N_D^-(v)$ to denote the set of all out-neighbors and in-neighbors of $v$, respectively, in $D$.
The \emph{in-degree} of a vertex $v$ in $D$, denoted by
$\deg_D^-(v)$, is the number of in-neighbors of $v$.
Similarly the \emph{out-degree} of a vertex $v$ in $D$, denoted by $\deg_D^+(v)$, is the number of out-neighbors of $v$.

For two digraphs $D=(V,E)$ and $D'=(V',E')$,
a bijection $f:V\to V'$ is called an \emph{isomorphism}
when
$(u,v)\in E$ if and only if $(f(u),f(v))\in E'$.
Two digraphs are \emph{isomorphic} if there is an isomorphism.

For a digraph $D=(V,E)$ with a fixed ordering $\{v_1,v_2,\ldots,v_n\}$ of $V$,
the \emph{adjacency matrix} of $D$ is an $n\times n$ matrix
$A_D=(a_{ij})_{i,j\in \{1,2,\ldots,n\}}$
such that
\[
a_{ij} =              \begin{cases}
  1 & \text{if $(v_i, v_j) \in E$,} \\
  0 & \text{otherwise.}
\end{cases}
\]

Let $\mathcal{G}_n$ be the set of all acyclic digraphs on the vertex set
$\{1,2,\ldots,n\}$ and let $\mathcal{A}_n$ be the set of all adjacency
matrices of acyclic digraphs in $\mathcal{G}_n$.

\subsection{Bott equivalence of acyclic digraphs}
An acyclic digraph $H$ is
\emph{Bott equivalent} to an acyclic digraph $D$
if $H$ has an isomorphic digraph that is
obtained from $D$
by successively applying local complementations
and slides.
In the following, we define local complementations and slides.

For a vertex $v$ of a digraph $D$, let $D*v$ be a digraph obtained by
adding an arc $(u,w)$ if $(u,w)\notin E$, or removing the arc
$(u,w)$ otherwise, for each pair $(u,w)\in N_D^-(v)\times N_D^+(v)$
with $u\neq w$.  This operation to obtain $D*v$ from $D$ is called a
\emph{local complementation} at $v$.  This operation is not new; As
far as the authors find, Bouchet~\cite{Bouchet1987c} wrote one of the
earliest papers on local complementations on digraphs.

For two sets $X$ and $Y$, we write $X\Delta Y=(X\setminus Y)\cup
(Y\setminus X)$.
For two distinct vertices $v, w$ having the same set of in-neighbors in
a digraph $D$,
we define $D\slide vw$ to be a digraph obtained by replacing
$N_D^+(w)$ by $N_D^+(w)\Delta N_D^+(v)$.
This operation to obtain $D\slide vw$ from $D$ is called a \emph{slide}. See Figure~\ref{fig:local complement and slides}.

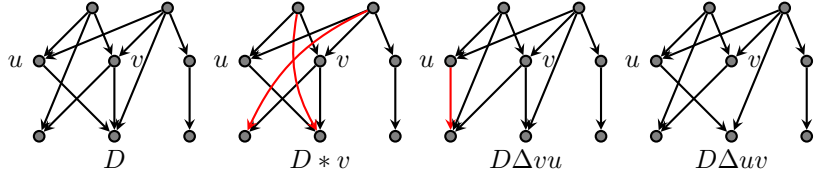
\begin{figure}
  \centering
  \tikzstyle{v}=[circle, draw, solid, fill=black!50, inner sep=0pt, minimum width=4pt]
  \tikzstyle{every edge}=[->,>=stealth,draw]
 \begin{tikzpicture}[thick,scale=0.5]
    \node [v] (bot1) {};
    \node [v] (bot2) [right of=bot1]{};
    \node [v] (bot3) [right of=bot2]{};
    \node[v] (mid1) [above of=bot1][label=left:$u$]   {}
    edge[->] (bot2);
    \node[v] (mid2) [right of=mid1][label=right:$v$] {}
    edge[->] (bot1)
    edge[->] (bot2);
   \node [v](mid3) [right of=mid2] {}
    edge[->] (bot3);
    \node [v] (top1) [above right of=mid1] {}
    edge[->] (bot1)
    edge[->] (mid1)
    edge[->] (mid2);
    \node [v] (top2) [right of=top1] {}
    edge[->] (mid1)
    edge[->] (mid2)
    edge[->] (mid3)
    edge[->] (bot2);
    \node [below of=bot2,node distance=3mm] {$D$};
\end{tikzpicture}
 \begin{tikzpicture}[thick,scale=0.5]
    \node [v] (bot1) {};
    \node [v] (bot2) [right of=bot1]{};
    \node [v] (bot3) [right of=bot2]{};
    \node[v] (mid1) [above of=bot1][label=left:$u$]   {}
    edge[->] (bot2);
    \node[v] (mid2) [right of=mid1][label=right:$v$] {}
    edge[->] (bot1)
    edge[->] (bot2);
   \node [v](mid3) [right of=mid2] {}
    edge[->] (bot3);
    \node [v] (top1) [above right of=mid1] {}
    edge[->,bend angle=20,bend right,color=red] (bot2)
    edge[->] (mid1)
    edge[->] (mid2);
    \node [v] (top2) [right of=top1] {}
    edge[->] (mid1)
    edge[->] (mid2)
    edge[->] (mid3)
    edge[->,bend angle=20,bend right,color=red] (bot1);
    \node [below of=bot2,node distance=3mm] {$D*v$};
\end{tikzpicture}
 \begin{tikzpicture}[thick,scale=0.5]
    \node [v] (bot1) {};
    \node [v] (bot2) [right of=bot1]{};
    \node [v] (bot3) [right of=bot2]{};
    \node[v] (mid1) [above of=bot1][label=left:$u$]   {}
    edge[->,color=red] (bot1);
    \node[v] (mid2) [right of=mid1][label=right:$v$] {}
    edge[->] (bot1)
    edge[->] (bot2);
   \node [v](mid3) [right of=mid2] {}
    edge[->] (bot3);
    \node [v] (top1) [above right of=mid1] {}
    edge[->] (bot1)
    edge[->] (mid1)
    edge[->] (mid2);
    \node [v] (top2) [right of=top1] {}
    edge[->] (mid1)
    edge[->] (mid2)
    edge[->] (mid3)
    edge[->] (bot2);
   \node [below of=bot2,node distance=3mm] {$D\Delta vu$};
\end{tikzpicture}
 \begin{tikzpicture}[thick,scale=0.5]
    \node [v] (bot1) {};
    \node [v] (bot2) [right of=bot1]{};
    \node [v] (bot3) [right of=bot2]{};
    \node[v] (mid1) [above of=bot1][label=left:$u$]   {}
    edge[->] (bot2);
    \node[v] (mid2) [right of=mid1][label=right:$v$] {}
    edge[->] (bot1);
  \node [v](mid3) [right of=mid2] {}
    edge[->] (bot3);
    \node [v] (top1) [above right of=mid1] {}
    edge[->] (bot1)
    edge[->] (mid1)
    edge[->] (mid2);
    \node [v] (top2) [right of=top1] {}
    edge[->] (mid1)
    edge[->] (mid2)
    edge[->] (mid3)
    edge[->] (bot2);
   \node [below of=bot2,node distance=3mm] {$D\Delta uv$};
 \end{tikzpicture}
 \caption{Two operations: a local complementation and a slide}
 \label{fig:local complement and slides}
\end{figure}

It is easy to observe that if $D$ is an acyclic digraph, then
so are $D*v$ and $D\slide xy$ (assuming $N_D^-(x)=N_D^-(y)$).
Furthermore $D*v*v=D$ and $D\slide xy\slide xy=D$.

\section{Small covers over cubes} \label{sec:small cover}
\subsection{Small covers over cubes}
Let $P$ be a simple polytope of dimension $n$ and let $\mathcal{F}(P)=\{F_1 , \ldots, F_m\}$ be the set of facets, codimension one faces, of $P$. It is known by \cite{Da-Ja-1991} that all small covers over $P$ can be classified by certain maps $\lambda:\mathcal{F}(P) \rightarrow \Z_2^n$ satisfying the so-called \emph{non-singularity condition}; $\{ \lambda(F_{i_1}) , \ldots , \lambda(F_{i_n}) \}$ is a basis of $\Z_2^n$ whenever the intersection $F_{i_1} \cap \cdots \cap F_{i_n}$ is non-empty. We call $\lambda$ a \emph{characteristic function}.

The construction of the small cover with respect to $\lambda$ is simple. Let $S^0(F_i)$ be the subgroup of $T (= (S^0)^n)$ generated by $\lambda(F_i)$. Given a point $p \in P$, we denote by $G(p)$ the minimal face containing $p$ in its relative interior. Assume $G(p) = F_{j_1} \cap \cdots \cap F_{j_k}$. Then $S^0(G(p)) = \oplus_{i=1}^{k} S^0(F_{j_i})$. Note that $S^0(G(p))$ is a $k$-dimensional subgroup of $T$. Let $M(\lambda)$ denote
\begin{equation}\label{eq:construction of M(lambda)}
M(\lambda)=P \times T /\sim,
\end{equation}
where $(p,g) \sim (q,h)$ if $p=q$ and $g^{-1}h \in S^0(G(p))$.
The canonical free action of $T$ on $P \times T$ descends to an action on $M(\lambda)$. It is easy to show that the action is locally standard and the orbit space can be identified with $P$. Thus $M(\lambda)$ is a small cover over $P$.

Two small covers $M$ and $N$ over $P$ are \emph{Davis-Januskiewicz equivalent} (or simply, \emph{D-J equivalent}) if there is a weak T-equivariant homeomorphism $f \colon M \to N$ covering the identity on $P$. By \cite{Da-Ja-1991}, $M(\lambda_1)$ is D-J equivalent to $M(\lambda_2)$ if and only if there is an automorphism $\sigma \in \text{Aut}(\Z_2^n)$ such that $\lambda_1 = \sigma \circ \lambda_2$.

When $P$ is an $n$-dimensional cube, $P$ has $2n$ facets, and for each facet $F$ of $P$ there exists a unique facet which does not intersect with $F$. So we give an ordering of the facets such that the facets $F_j$ and $F_{n+j}$ do not intersect for $1 \leq j \leq n$. Then we may assign an $n \times 2n$ matrix to a characteristic function $\lambda$ by choosing a basis for $(\Z_2)^n$. Then
\[
\left( \lambda(F_1) \cdots \lambda(F_{2n}) \right) = ( A | B ),
\]
where $A$ and $B$ are $n \times n$ matrices. Since
$\bigcap_{j=1}^{n}F_j \neq \emptyset$, we also may assume that $A$ is the identity matrix of size $n$. Then we can see that the
non-singularity condition of $\lambda$ is equivalent to the
following; \emph{every principal minor of $B$ is 1}.

Hence, there is a bijection between the DJ-equivalence classes of small covers over cubes and the set $M(n)$ of $n \times n$ matrices over $\Z_2$ all of whose principal minors are 1.

\begin{theorem}[Choi {\cite[Theorem 2.2]{Choi-2008}}]
  Let $\phi : \mathcal{A}_n \to M(n)$ by $A \mapsto I_n + A^t$, where $I_n$ is the identity matrix of size $n$ and $A^t$ is the transpose matrix of $A$. Then $\phi$ is a bijection.
\end{theorem}

Let $D \in \mathcal{G}_n$ and let $A_D$ be its adjacency matrix. By the above
theorem, $(I_n | \phi(A_D))$ is the characteristic function. For the
sake of simplicity, we denote the corresponding small cover by
$M(A_D)$.

Let $D, H \in \mathcal{G}_n$. We note that the ordering of vertices of an acyclic digraph is nothing but the ordering of facets of $P$, and the diffeomorphism type of the small cover is independent on the ordering of facets by the construction \eqref{eq:construction of M(lambda)}. Hence, if $D$ and $H$ are isomorphic, then $M(A_D)$ and $M(A_H)$ are diffeomorphic (but not D-J equivalent).

\subsection{Real Bott manifolds}
To define a real Bott manifold, we consider with a sequence of $\RP^1$ bundles
\begin{equation}
    M_n \stackrel{\pi_n}\longrightarrow M_{n-1} \stackrel{\pi_{n-1}}\longrightarrow \cdots \stackrel{\pi_2}\longrightarrow M_1 \stackrel{\pi_1}\longrightarrow M_0 =\{\text{a point} \}, \label{eqn:Bott tower}
\end{equation}
where each fiber bundle $\pi_j : M_j \to M_{j-1}$ for $j=1, \ldots, n$ is the projectivization of a Whitney sum of a real line bundle and the trivial real line bundle over $M_{j-1}$. We call a sequence \eqref{eqn:Bott tower} a \emph{real Bott tower} and $M_n$ a \emph{real Bott manifold}. It is well-known that real line bundles are classified by their first Stiefel-Whitney classes and $H^1(M_{j-1}; \Z_2)$ is isomorphic to $(\Z_2)^{j-1}$. Hence, each $M_j$ is determined by a vector $A_j$ in $(\Z_2)^{j-1}$. One can organize this vector into an $n \times n$ upper-triangular square matrix $A$ with zero diagonals by putting $A_j$ as the upper part of the $j$-th column vector.

\begin{proposition}[Panov and Masuda {\cite[Proposition 3.1]{Ma-Pa-2008}}]
Let $A$ be an upper-triangular square matrix with zero diagonals. A real Bott tower with respect to $A$ carries a natural real torus action turning it into a small cover $M(A)$\footnote{The original statement of \cite[Proposition 3.1]{Ma-Pa-2008} is written in terms of complex versions of a Bott tower and a small cover called a \emph{quasitoric manifold}. But the same argument can be easily applied to the real cases.}.
\end{proposition}

If $D$ is an acyclic digraph, then its adjacency matrix $A_D$ with respect to its acyclic ordering is an upper-triangular square matrix with zero diagonals. Hence, all small covers over a cube are diffeomorphic to real Bott manifolds.

\section{Classification of real Bott manifolds} \label{sec:classification}

For a matrix $A$, we write  $A_j$ to denote the $j$-th column of $A$
and write $A^i_j$ to denote its $(i,j)$-th entry of $A$.
Masuda \cite{Masuda-2008} defined three operations $\Phi_S$,
$\Phi^k$ and $\Phi^I_C$ on $\mathcal{A}_n$ as follows:
(Notice that these matrices are over $\Z_2$ and therefore $1+1=0$.)
\begin{enumerate}
\item $\Phi_S(A) = SAS^{-1}$, where $S$ is a permutation matrix
\item For $k \in \{1, \ldots, n\}$,
  \[\Phi^k(A)_j = A_j + A^{k}_j A_k \quad\text{for $j = 1, \ldots, n$}.\]
\item Let $I=\{i_1,i_2,\ldots,i_\ell\}$ be a subset of $\{1, \ldots, n\}$ such that $A_i = A_j$ for $i,j \in I$ and $i_1<i_2<\cdots<i_\ell$.
  Let $C$ be an invertible $|I|\times |I|$ matrix over $\Z_2$. Then,
  \[\Phi^I_C(A)^i_j = \begin{cases}
    \sum_{k=1}^{|I|} C^{m}_{k} A^{i_k}_j & \text{if $i=i_m\in I$,} \\
    A^i_j & \hbox{if $i \not\in I$.}
  \end{cases}
  \]
Roughly speaking, $\Phi^I_C$ is the left matrix multiplication of $C$ to the submatrix of $A$ which consists of $i$-th rows for $i \in I$.
\end{enumerate}

\begin{remark}
  Indeed, Masuda \cite{Masuda-2008}
  described all three operations on square upper-triangular matrices over
  $\Z_2$.
  But since every
  small cover over a cube is diffeomorphic to a real Bott manifold,
  these three operations can be extended to $\mathcal A_n$.
  Moreover, this is more natural because the upper-triangular matrices
  are
  not closed under $\Phi_S$.
\end{remark}

We say that two matrices in $\mathcal{A}_n$ are \emph{Bott equivalent}
if one is transformed to the other through a sequence of the three
operations.
Masuda \cite{Masuda-2008} proved that Bott equivalence on $\mathcal
A_n$ characterizes real Bott manifolds up to diffeomorphisms.
\begin{theorem}[{Masuda \cite[Theorem 1.1]{Masuda-2008}}]\label{thm:masuda}
  Two matrices $A, B\in \mathcal A_n$ are Bott equivalent if and only if corresponding real Bott manifolds $M(A)$ and $M(B)$ are diffeomorphic.
\end{theorem}

We now prove our main observation, stating that two acyclic digraphs $D$ and $H$ are Bott equivalent
(in the sense of the definition in Section \ref{section:acyclic digraphs})
if and only if $A_D$ and $A_H$ are Bott equivalent.

\begin{lemma}\label{lem:bottequ}
  Two acyclic digraphs are Bott equivalent
  if and only if their adjacency matrices are Bott equivalent.
\end{lemma}
\begin{proof}
  Let $D=(V,E)$ be an acyclic digraph
  with a certain ordering $v_1,v_2,\ldots,v_n$ of the vertices.
  Obviously $A_{D*v_k}=\Phi^k(A_D)$
  and
  if $i<j$, then
  $A_{D\slide v_iv_j}=\Phi_C^I(A_D)$
  with $C=
  \left(\begin{smallmatrix}
      1 & 0 \\ 1 & 1
    \end{smallmatrix}\right)
  $.
  If $i>j$, then
  $A_{D\slide v_iv_j}=\Phi_C^I(A_D)$
  with $C=
  \left(\begin{smallmatrix}
      1 & 1 \\ 0 & 1
    \end{smallmatrix}\right)
  $.
 Therefore Bott equivalent acyclic digraphs have Bott equivalent adjacency matrices.

  Now, let us prove the converse.
  We want to show that
  each of three operations $\Phi_S$, $\Phi^k, \Phi_C^I$
  on the adjacency matrix $A_D$ of $D$
  gives an adjacency matrix of another acyclic digraph Bott equivalent to $D$.
  Firstly, $\Phi_S$ will simply give an  isomorphic digraph, which is again Bott equivalent to $D$.
  Secondly, $\Phi^k(A_D)=A_{D*v_k}$ and therefore $\Phi^k$ corresponds to a local complementation at $v_k$.

  Last, we claim that $\Phi_C^I$ can be expressed as a sequence of slides.
  Suppose that this claim is false.
  Since $C$ is invertible,
  there is a sequence of elementary row operations $\sigma_1,\sigma_2,\ldots,\sigma_t$
  so that $C=\sigma_1\circ \sigma_2\circ \cdots\circ \sigma_t(I)$.
  Since $C$ is   a matrix over $\Z_2$, each $\sigma_i$
  is either an operation to interchange two rows of $C$
  or an operation to add a row to another row in $C$.
  Let $C'=\sigma_2\circ \cdots \circ\sigma_t(I)$. Then $C'$ is invertible
  and $C=\sigma_1(C')$. We assume that $C$ is chosen as a counterexample with minimum $t$.
  Thus, $\Phi_{C'}^I(A_D)=A_{D'}$ for some acyclic digraph $D'$
  Bott equivalent to $D$.
 If $\sigma_1$ is an operation to interchange two rows,
  then $  \Phi_{C'}^I(A_{D'})$ is an adjacency matrix after swapping two corresponding vertices in the fixed order, because these two vertices have the same set of in-neighbors
  and therefore clearly $\Phi_{C'}^I(A_{D'})$ is an adjacency matrix of a digraph isomorphic to $D'$.

  If $\sigma_1$ is an operation to add a row to another,  then we can simulate $\sigma_1$ by a slide on $D'$ and therefore
  $\Phi_{C}^I(A_D)= A_{D'\slide uv}$ for two vertices $u, v\in I$.
 This proves the claim because $D'\slide uv$ is Bott equivalent to $D$.
\end{proof}

Now we are ready to see the following theorem, implied by Theorem~\ref{thm:masuda} and Lemma~\ref{lem:bottequ}.
\begin{theorem} \label{theorem:topological classification} Two
  acyclic digraphs $D$ and $H$ are Bott equivalent if and only if the
  small covers over cubes $M(A_D)$ and $M(A_H)$ are diffeomorphic.
\end{theorem}

\section{Invariants of real Bott manifolds} \label{sec:invariant}
In this section, we introduce topological invariants of small covers
over cubes. To do so,
we focus on finding parameters on acyclic digraphs
that are invariant under taking local complementations and slides.

A vertex in an acyclic digraph is called a \emph{root} if it has no in-neighbor.
Let $L_k(D)$ be the set of vertices $v$ such that a longest directed path ending at $v$ has exactly $k$ edges.
Each $L_k(D)$ is called the $k$-th level of $D$.
Clearly $L_0(D)$ is the set of roots.
For an acyclic digraph $D$ with $n$ vertices,
a \emph{level   sequence} of $D$
is a sequence
$(|L_0(D)|,|L_1(D)|,|L_2(D)|,\ldots,|L_{n-1}(D)|)$.
It is obvious that $\cup_{i=0}^{n-1}L_i(D)=V(D)$.
Note that for each vertex $v\in L_i(D)$,
all in-neighbors of $v$ are in $\cup_{j=0}^{i-1} L_j(D)$
and all out-neighbors of $v$ are in $\cup_{j=i+1}^{n-1}L_j(D)$.
\begin{proposition}\label{prop:level}
  Bott equivalent acyclic digraphs have the identical level sequence.
\end{proposition}
\begin{proof}
  Let $D$ be an acyclic digraph.
  It is enough to show that both local complementations and slides do not change $L_i(D)$.
  Let $v$ be a vertex in $L_i(D)$. Then there is a longest directed path $P_v$ in $D$  ending at $v$ with exactly $i$ edges.

  Let us first show that $v\in L_i(D*w)$  for a vertex $w$ in $D$.
 It is enough to show that $P_v$ is a path in $D*w$ as well, because, if so, then  a longest path in $D*w$ ending at $v$ will be a path in $D*w*w=D$ as well.
  Suppose that $P_v$ is not a path in $D*w$.
  Then $w$ must remove at least one arc $(x,y)$ of $P_v$ and therefore $(x,w)$ and $(w,y)$ are arcs of $D$.
  Since $D$ is acyclic, $w$ is not on $P_v$. Then by replacing the arc $(x,y)$ by a path $xwy$ in $P_v$, we can find a path longer than $P_v$ in $D$, contradictory to the assumption that $P_v$ is a longest path ending at $v$.
  This proves the claim that $L_i(D)=L_i(D*w)$ for all $i$.

  Now let us prove that $v\in L_i(D\slide uw)$ for two vertices $u,w$ having the set of in-neighbors. Again, it is enough to show that $D\slide uw$ has a path of length $i$ ending at $v$;
  because if $D\slide uw$ has  a longer path ending at $v$,
  then so does $D$ by the fact that $(D\slide uw)\slide uw=D$.
  We may assume that  $D\slide uw$ removes at least one arc $(x,y)$ of $P_v$.
  Then $w=x$ and both $(w,y)$ and $(u,y)$ are arcs of $D$.
  Since $u$ and $w$ have the same set of in-neighbors,
  we can replace $w$ by $u$ in $P_v$
  to obtain a path of the same length in $D\slide uw $.
  This completes the proof.
\end{proof}

The \emph{rank} of a digraph $D$ is the rank of the adjacency matrix
$A_D$ over $\Z_2$.
\begin{proposition}\label{prop:rank}
  Bott equivalent acyclic digraphs have
  the same rank.
\end{proposition}
\begin{proof}
  Let $D$ be an acyclic digraph.
  For a vertex $v$,
  we can obtain $A_{D*v}$ from $A_D$
  by adding the row of $v$ to rows of in-neighbors of $v$.
  Note that since $D$ is acyclic, no in-neighbor of $v$ is an out-neighbor of $v$.
  Therefore $\rank(A_{D*v})=\rank(A_D)$.

  Similarly
  let $u,w$ be two distinct vertices of $D$ having the same set of in-neighbors.
  Clearly $D$ and $D\slide uw$ have the same rank
  because
 $A_{D\slide uw}$ is obtained by adding the row of $u$ to the row of $w$ in $A_D$
  and therefore $\rank(A_{D\slide uw})=\rank (A_D)$.
\end{proof}

For subsets $X$ and $Y$ of the vertex set of a digraph $D$,
we write $A_D[X,Y]$ to denote the submatrix of $A_D$ whose rows
correspond to $X$ and columns correspond to $Y$.
We write $\rho_D(X,Y)=\rank(A_D[X,Y])$
where $A_D[X,Y]$ is considered as a matrix over $\Z_2$.
We write $\rho_D(X)=\rho_D(X,V(D)\setminus X)$.

\begin{proposition}\label{prop:cutrank}
  Let $D$, $H$ be Bott equivalent acyclic digraphs on $n$ vertices.
  Then,
 \begin{enumerate}[(i)]
  \item   $\rho_D(\cup_{i\in I} L_i(D))=\rho_H(\cup_{i\in I} L_i(H))$
    for every subset $I$ of $\{0,1,2,\ldots,n-1\}$,
  \item  $\rho_D(L_i(D),L_{i+1}(D))=\rho_H(L_i(H),L_{i+1}(H))$ for all
    $i\in \{0,1,2,\ldots,n-2\}$.
  \end{enumerate}
\end{proposition}
\begin{proof}
  It is enough to prove when $H=D*v$ or $H=D\slide uw$.
  By Proposition \ref{prop:level},
  $L_i(D)=L_i(H)$ for each $i$.

  (i)
  For a subset $I=\{0,1,2,\ldots,n-1\}$,
  let $X=\cup_{i\in I} L_i(D)$
  and $Y=V(D)\setminus X$.
  Let $M=A_D[X,Y]$
  and $M'=A_{H}[X,Y]$.
  Then $\rho_D(X)=\rank M$
  and $\rho_H(X)=\rank M'$.

  Let us first consider the case when $H=D*v$ for a vertex $v$ of $D$.
  Then either $v\in X$ or $v\in Y$.
  If $v\in X$, then $M$ has a row indexed by $v$
  and $M'$ is obtained from $M$
  by adding the row of $v$ to all rows indexed by in-neighbors of $v$ in $X$.
 If $v\in Y$, then $M$ has a column indexed by $v$
  and $M'$ is obtained from $M$
  by adding the column of $v$ to all columns indexed by out-neighbors of $v$ in $Y$.
  Thus, in both cases, $\rank M=\rank M'$.

  Now let us consider the case when $H=D\slide uw$ for two vertices
  $u$, $w$ having the same set of in-neighbors.
  Since $u$ and $w$ have the same set of in-neighbors, they belong to the same level.
  Therefore either $\{u,w\}\subseteq X$ or $\{u,w\}\subseteq Y$.
  If $\{u,w\}\subseteq Y$, then $M'=M$.
  If $\{u,w\}\subseteq X$, then $M'$ is obtained from $M$ by adding the row of $u$
  to the row of $w$.
  So $\rank M=\rank M'$.
 This completes the proof of (i).

 (ii)
 Since $D$ is acyclic,
 there is no arc from $L_{a}(D)$ to $L_b(D)$ if $a>b$.
 Therefore
 \[
 \rho_{D}(L_i(D),L_{i+1}(D))=\rho_D(L_i(D)\cup L_{i+2}(D)\cup L_{i+3}(D)\cup \cdots \cup L_{n-1}(D)),\]
 and by (i),
 we have $\rho_D(L_i(D),L_{i+1}(D))=\rho_H(L_i(H),L_{i+1}(H))$.
\end{proof}

For two vertices $x$ and $y$ of $D$,
we say that $x\sim_D y$ if $x$ and $y$ have the same set of
  in-neighbors.
  Then $\sim_D$ is an equivalence relation of $V(D)$
  and each equivalence class is called a \emph{sibling group} of $D$.
A sibling group  of $D$
corresponds to a maximal set of identical columns in $A_D$.
Clearly each level $L_i(D)$ is partitioned into sibling groups.
\begin{proposition}
  If $D$ and $H$ are Bott equivalent acyclic digraphs,
  then
  $L_i(D)$
  and $L_i(H)$ have the same number of sibling groups having exactly $k$ vertices
  for each pair of $i$ and $k$.
\end{proposition}
\begin{proof}
We may assume that either $H=D*v$ for a vertex $v$
  or $H=D\slide uw$ for two vertices $u$, $w$ having the same set of
  in-neighbors.

 We claim that $\sim_D$ and $\sim_H$ are identical equivalence relations.
 To do so, we will prove that if $x$ and $y$ have the same set of in-neighbors in $D$,
  then they have the same set of in-neighbors in $H$.

 Firstly let us consider the case when $H=D*v$ for a vertex $v$.
  We may assume that $(v,x)$ or $(v,y)$ is an arc of $D$, because otherwise
  the set of in-neighbors of $x$ or $y$ will not be changed by
  applying a local complementation at $v$.
  Since $x$ and $y$ have the same set of  in-neighbors, both $(v,x)$
  and $(v,y)$ are arcs of $D$. Then $x$ and $y$ still have
  the same set of in-neighbors in $D*v$.

  Secondly let us assume that $H=D\slide uw$.
  Similarly we may assume that both $(u,x)$ and $(u,y)$ are arcs of
  $D$.
  Since $x$ and $y$ have the same set of in-neighbors,
  $(w,x)$ is an arc of $D$ if and only if $(w,y)$ is an arc of $D$.
  After a slide, $(w,x)$ is an arc of $H$ if and only if $(w,x)$ is
  not an arc of $D$
  and similarly $(w,y)$ is an arc of $H$ if and only if $(w,y)$ is not
  an arc of $D$. This proves the claim.
\end{proof}

The \emph{odd height} of an acyclic digraph $D$
is the length of a longest directed path ending at a vertex of odd out-degree.
In other words, it
is the maximum $k$ such that $L_k$ contains a vertex of odd
out-degree.
If $D$ has no vertex of odd out-degree, then we assume that its odd
height is $\infty$.

\begin{proposition}
  Bott equivalent acyclic digraphs have the same odd height.
\end{proposition}
\begin{proof}
  Let $D$ be an acyclic digraph.
 Let $v$ be a vertex of $D$. Let $x,y$ be two distinct vertices
  having the same set of in-neighbors. Then
  \begin{align*}
    \deg^+_{D*v}(w)& \equiv
    \begin{cases}
      \deg^+_{D}(w)+\deg^+_D(v)&\pmod2   \quad\text{ if $w$ is an in-neighbor of $v$,}\\
      \deg^+_D(w)&\pmod 2 \quad\text{ otherwise.}
    \end{cases}
    \\
    \deg^+_{D\slide xy}(w)&\equiv
    \begin{cases}
      \deg^+_D(w)+\deg^+_D(x)&\pmod 2 \quad\text{ if }w=y,\\
      \deg^+_D(w)&\pmod 2 \quad\text{ otherwise.}
    \end{cases}
  \end{align*}
  If every vertex of $D$ has even out-degree,
  then clearly a local complementation and a slide do not create a vertex of
  odd out-degree.
  So we may assume that $D$ has a vertex of odd out-degree.
  Let $k$ be the odd height of $D$.
  Let $w$ be a vertex of odd out-degree in $L_k(D)$.

  Let us first consider $D*v$ for a vertex $v\in L_i(D)$.
  If $w$ is an in-neighbor of $v$ in $D$, then
  the out-degree of $w$ in $D*v$ is odd because
  $i>k$
  and $\deg^+_D(v)$ is even.
  If $w$ is not an in-neighbor of $v$ in $D$, then
  the out-degree of  $w$ in $D*v$ is again odd.
  This proves that the odd height of $D*v$ is at least the odd height
  of $D$. Since $(D*v)*v=D$, we conclude that $D$ and $D*v$ have the
  same odd height.

  Now we claim that $L_k(D\slide xy)$ has a vertex of odd out-degree in
  $D\Delta xy$.
  Suppose that the out-degree of $w$ in $D\slide xy$ is even.
  Then $w=y$, $x\in L_k(D)$, and the out-degree of $x$ in $D$ is odd.
  Since $(D\slide xy)\slide xy=D$, we conclude that $D$ and $D\slide
  xy$ have the same odd height.
\end{proof}

We observe that the above invariants have topological motivations.
For an acyclic digraph $D$, let $M(A_D)$ be the small cover
corresponding to $D$. Choi \cite{Choi-10} showed that $M(A_D)$ is orientable if and only if the out-degree of every vertex of $D$ is even; in other words, the odd height of $D$ is $\infty$. Hence, the odd height can be seen as a generalization of the orientability of real Bott manifolds. In addition, Ishida \cite{Ishida-10} proved that $M(A_D)$ is symplectic if and only if $V(D)$ can be partitioned into pairs of vertices so that two vertices in a pair have the same set of in-neighbors. Equivalently, the cardinality of each sibling group is even. Hence the number of sibling groups with a fixed size is also a generalization of the symplecticness of real Bott manifolds. It is easy to see that if the cardinality of each sibling group is even, then its odd height must be $\infty$. It is obvious with topological viewpoint as well, because every symplectic manifold is orientable.
We computed a few values of the numbers of $n$-dimensional orientable real Bott manifolds and symplectic manifolds up to diffeomorphism in Table~\ref{tab:Bott equivalence}.

We remark that the invariants discussed in this paper completely
classify all Bott equivalence classes up to $4$ vertices but not on  $5$ vertices.
One can easily check that two acyclic digraphs in Figure~\ref{fig:unfortune
  case} are not Bott equivalent but have the same set of invariants.
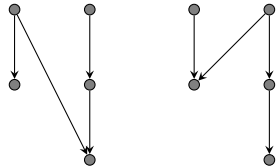
\begin{figure}
  \centering
  \tikzstyle{v}=[circle, draw, solid, fill=black!50, inner sep=0pt, minimum width=4pt]
  \tikzstyle{every edge}=[->,>=stealth,draw]
  \begin{tikzpicture}
    \node [v] (r3) {};
    \node [v] [above of=r3] (r2) {}
    edge[->] (r3);
    \node [v] [above of=r2] (r1) {}
    edge[->] (r2);
    \node [v] [left of=r2] (l2) {};
    \node [v] [above of=l2] (l1) {}
    edge [->] (l2)
    edge [->] (r3);
  \end{tikzpicture}
  \hspace{1cm}
  \begin{tikzpicture}
    \node [v] (r3) {};
    \node [v] [above of=r3] (r2) {}
    edge[->] (r3);
   \node [v] [left of=r2] (l2) {};
    \node [v] [above of=l2] (l1) {}
    edge [->] (l2);
    \node [v] [above of=r2] (r1) {}
    edge[->] (r2)
    edge[->] (l2);
  \end{tikzpicture}
  \caption{Theses acyclic digraphs are not Bott equivalent but have the identical set of invariants discussed in this paper}
  \label{fig:unfortune case}
\end{figure}

\section*{Acknowledgment}
We would like to thank B. D. McKay.
His computer program called \emph{nauty} to test graph isomorphism
has been very useful for our enumeration.
We also used the list of
non-isomorphic acyclic
digraphs up to 8 vertices posted on his website.

\end{document}